\documentclass[11pt]{amsart}
\usepackage{geometry}                
\usepackage{graphicx}
\usepackage{amssymb}
\usepackage{epstopdf}
\usepackage{amssymb, amsmath, amscd, amsthm,
color, epsfig,enumerate
}
\usepackage[all]{xy}          
\xyoption{dvips}              

\title{Homology of non-$k$-overlapping discs}
\author{Natalya Dobrinskaya}
\address{}
\email{ne.dobrinsk@gmail.com}
\author{Victor Turchin}
\address{Department of Mathematics\\ Kansas State University \\ Manhattan, KS 66506 \\ USA}
\email{turchin@ksu.edu}
\date{}                                           

\newcommand{\calB}{\mathcal{B}}

\newcommand{\calO}{\mathcal{O}}
\newcommand{\calF}{\mathcal{F}}

\newcommand{\calR}{\mathcal{R}}
\newcommand{\calM}{\mathcal{M}}
\newcommand{\calH}{\mathcal{H}}

\newcommand{\R}{{\mathbb R}}
\newcommand{\Z}{{\mathbb Z}}
\newcommand{\Q}{{\mathbb Q}}

\newcommand{\sEmb}{\operatorname{sEmb}}
\newcommand{\Emb}{\operatorname{Emb}}
\newcommand{\Ebar}{\overline{\Emb}}
\newcommand{\Imm}{\operatorname{Imm}}
\newcommand{\Ibar}{\overline{\Imm}}

\newcommand{\Tot}{\operatorname{Tot}}
\newcommand{\id}{\mathrm{id}}

\newcommand{\one}{{\mathtt 1}}

\newcommand{\hRmod}{\operatorname{hRmod}}

\newcommand{\hIBim}{\operatorname{hIBim}}
\newcommand{\hOperad}{\operatorname{hOperad}}
\newcommand{\hBim}{\operatorname{hBim}}

\usepackage{graphicx, psfrag,rotating}

\newcommand{\Assoc}{{\mathcal A}ssoc}

\newcommand{\Com}{{\mathcal C}om}
\newcommand{\Lie}{{\mathcal L}ie}

\newcommand\rth{\refstepcounter{equation}}
\newcommand\numb{\rth{\rm \theequation}}
\numberwithin{equation}{section}

\theoremstyle{plain}
\newtheorem{theorem}{Theorem}[section]
\newtheorem{proposition}[theorem]{Proposition}
\newtheorem{lemma}[theorem]{Lemma}
\newtheorem{corollary}[theorem]{Corollary}

\theoremstyle{definition}
\newtheorem{definition}[theorem]{Definition}
\newtheorem{example}[theorem]{Example}

\theoremstyle{remark}
\newtheorem{remark}[theorem]{Remark}

\begin{document}

\begin{abstract}
In this paper we describe the homology and cohomology of some natural bimodules  over the little discs operad, whose components are configurations of non-$k$-overlapping discs. At the end we briefly explain how this algebraic structure intervenes in the study of spaces of non-$k$-equal immersions.
\end{abstract}

\maketitle


\section{Introduction}\label{s:intro}

Let $\calB_d$ denote the operad of little $d$-discs. We will consider the bimodules $\calB_d^{(k)}$, $k\geq 2$, over it, with the $n$-th component $\calB_d^{(k)}(n)$ being the configuration space of $n$ open discs (labeled by $1\ldots n$) in a unit disc satisfying the {\it non-$k$-overlapping condition}: intersection of any $k$ of them is empty. It is straightforward that $\calB_d^{(k)}$ is a bimodule over $\calB_d$. The left action is given by the maps
$$
\calB_d(n)\times \calB_d^{(k)}(m_1)\times \calB_d^{(k)}(m_2)\times\ldots\times \calB_d^{(k)}(m_n)\to
\calB_d^{(k)}(m_1+\ldots +m_n)
$$
that consist in replacing the $i$-th disc in $\calB_d(n)$  by a configuration of discs from $\calB_d^{(k)}(m_i)$. The right action is given by similar maps
$$
\calB_d^{(k)}(n)\times \calB_d(m_1)\times \calB_d(m_2)\times\ldots\times \calB_d(m_n)\to
\calB_d^{(k)}(m_1+\ldots +m_n).
$$
Obviously, in both cases the resulting configuration always satisfies the non-$k$-overlapping condition, thus both composition maps are well defined.

It is easy to see that the space $\calB_d^{(k)}(n)$ is homotopy equivalent to the complement in $(\R^d)^{\times n}$  to the union of subspaces
$$
A_I=\left\{ (x_1,\ldots,x_n)\in (\R^d)^{\times n} \, \left| \, x_{i_1}=\ldots=x_{i_k}\right.\right\},
$$
where $I=\{i_1\ldots i_k\}$ runs through all cardinality $k$ subsets of $\underline{n}=\{1\ldots n\}$. We denote this complement by $\calM_d^{(k)}(n)$. By taking the centers of the balls one gets a map $\calB_d^{(k)}(n)\to \calM_d^{(k)}(n)$ which is a homotopy equivalence.

The homology groups of $\calM_d^{(k)}(n)$ were first computed by Bj\"orner and Welker~\cite{BjorWelk}, see also~\cite{BjorWachs1,BjorWachs2}. The cohomology algebra $H^*\calM_2^{(k)}(n)$ was determined by Yuzvinsky~\cite{Yuzv}. The latter reference also produces a rational model for $\calM_2^{(k)}(n)$. Based on this model it was shown in~\cite{Miller} that the spaces $\calM_2^{(3)}(n)$, $n\geq 7$, have non-trivial Massey products and thus are not formal. The cohomology algebra $H^*\calM_1^{(k)}(n)$ was computed by Baryshnikov~\cite{Bar}. The symmetric group action on $H_*\calM_d^{(k)}(n)$ was computed by Sundaram and Wachs~\cite{SundWachs}. Even though the (co)homology of $\calM_d^{(k)}(n)$ is now well understood, few of the references  give a geometric description of this (co)homology. In fact only in the case $d=1$ one has a geometrical description of this homology given by the first author in~\cite{Dobr3} and a geometrical description of cohomology given by Baryshnikov in~\cite{Bar}. More precisely, in~\cite{BjorWelk,BjorWachs1,BjorWachs2} the authors use the Goresky-MacPherson formula that describes the homology of the complement to a subspace arrangement in terms of cohomology of certain posets (of strata in the arrangement). In case of  $\calM_d^{(k)}(n)$ one has to study the poset $\Pi_{n,k}$ of subsets of $\{1\ldots n\}$ whose cardinality is either one or $\geq k$. Yuzvinsky's method is also purely combinatorial -- it produces a rational model for $\calM_2^{(k)}(n)$ and more generally for any complement to a complex arrangement based on the combinatorics of the Goresky-MacPherson complex. Another  approach for the case $d=1$ appears in~\cite{PRW} that describes the homology over a field of more general diagonal arrangements in terms of the homology of monomial rings. Applied to the case of non-$k$-equal arrangements this approach produces the Betti numbers of $\calM_1^{(k)}$. Following this idea improved to integral coefficients and using homology algebra methods, an algebraic structure similar to the one studied in this paper for $d=1$ appeared in~\cite{Dobr3}.

In this paper we also give a more geometrical description of the cohomology algebra $H^*\calB_d^{(k)}(n)$. In particular we show that this algebra is quadratic which seems to be known only in two  cases: $k=2$~\cite{Arn,Coh}, and $d=1$~\cite{Bar}. (In the other cases $k\geq3$ and $d\geq 2$ the generators lie in different degrees, but as we said all relations still follow from quadratic ones.)  Since our description is very geometrical we hope it will help to understand better the rational homotopy type of $\calM_d^{(k)}(n)$, in particular to compute the Massey products for these spaces.

As a left module $H_*\calB_d^{(k)}$, $k\geq 3$, is generated by two elements $x_1\in H_0\calB_d^{(k)}(1)$ and $\{x_1,\ldots,x_k\}\in H_{(k-1)d-1}\calB_d^{(k)}(k)$, see Theorem~\ref{t:l_mod_main}. Notice that $\calB_d^{(k)}(1)\simeq *$ and $\calB_d^{(k)}(k)\simeq S^{(k-1)d-1}$. The elements $x_1$ and $\{x_1\ldots x_k\}$ are the generators of the correspoding homology groups. Explicitly this result means that the homology groups $H_*\calB_d^{(k)}(n)$ are spanned by certain products of iterated brackets. Such classes can be geometrically realized as products of spheres in $\calB_d^{(k)}(n)$. One should mention that such description of the homology in terms  of iterated brackets  is implicitly given in~\cite{Feicht}, where the author shows that the poset $\Pi_{n,k}$ is quasi-isomorphic to a poset of certain trees. Here one can see a connection to a work of Gaiffi~\cite{Gaiffi} that produces a general construction of a compactification of the complement to a subspace arrangement. In the case of $\calM_d^{(k)}(n)$ the strata  of the compactification are encoded exactly by the trees from~\cite{Feicht}. In fact Gaiffi's work can  be used to produce geometric cycles in the homology of the complement to any arrangement.

The structure of a bimodule over $H_*\calB_d$ that $H_*\calB_d^{(k)}$ has, not only it gives a very explicit geometric description of cycles that span this homology, but also  is important for applications. One application is  in the study of spaces of non-$k$-self-intersecting immersions. We describe briefly this connection in Section~\ref{s:application}. 

Another important application is in the study of the homology of iterated loop spaces of  fat wedges. 
 First examples of such computations go back to Lemaire's work~\cite{Lemaire} for single-loop spaces on fat
wedges of spheres, who computed its homology over a field. In~\cite{Dobr1,Dobr2,Dobr3} 
a more general problem for loops on fat wedges of arbitrary spaces is considered,
and the homology is computed via homology of diagonal arrangements with
algebraic structure similar to bimodule on $\calB_1^{(k)}$. 
 The long brackets $\{x_1\ldots x_k\}$ discussed above  correspond to  higher commutator products on loop homology induced by  Samelson products.  A similar description  of the homology of iterated $d$-loops on fat wedges must exist and  as we hope will be studied elsewhere.

\subsection{Notation}\label{ss:notation}

By a symmetric sequence we will understand a sequence of objects $M(n)$, $n\geq 0$, where each $M(n)$ is endowed with an action of the symmetric group $\Sigma_n$.  Alternatively and this will be useful sometimes for our arguments, we will understand a symmetric sequence as a functor from the category of finite sets whose morphisms are bijections. For example for a finite set $I$,  the corresponding configuration space (or its homology) whose points/discs are encoded by elements from $I$, will be denoted by    $\calM_d^{(k)}(I)$, $\calB_d^{(k)}(I)$, $H_*\calB_d^{(k)}(I)$. The permutation group of $I$ will be denoted by $\Sigma_I$. The cardinal of~$I$ will be denoted by~$|I|$. The set $\{1\ldots n\}$ will be denoted by $\underline{n}$.

All the homology and cohomology groups that we consider are taken with integral coefficients.

\subsection{Main results}\label{ss:main_results}

Our main result is Theorem~\ref{t:bimod_main} where we describe the $H_*\calB_d$-bimodule structure of $H_*\calB_d^{(k)}=H_*\calM_d^{(k)}$. Another important result --- we give a more natural
 description of the cohomology algebras $H^*\calB_d^{(k)}(n)$ as spaces spanned by cerain forests, see
 Sections~\ref{s:coh_space}-\ref{s:multiplicative}. Such description is nicely compatible with the structure of a cobimodule that $H^*\calB_d^{(k)}$ has, see Section~\ref{s:co}. As we have mentioned the spaces
 $\calM_d^{(k)}(n)$ were extensively studied. In particular to prepare this note we found very useful~\cite{Bar,SundWachs,Yuzv}. However, the presentation of this paper is self contained --- all the arguments and proofs are not formally relying on other results or computations. For which reason we hope it will also be of educational value.

\subsection{Acknowledgements}\label{ss:ackn}
The authors are grateful to the Universit\'e  Catholique de Louvain where they were both working in the Spring 2007 and where one of the main results Theorem~\ref{t:bimod_main} was obtained. A special  thank goes to I.~F\'elix and P.~Lambrechts for the invitation to work  at the UCL and for numerous  mathematical discussions. The second author  is grateful to the MPIM and IHES where he was working on the main details of the proofs. Finally the authors thank V.~Vassiliev for asking a question to which Theorem~\ref{t:cosimpl} is an answer, and J.~Mostovoy for encouraging to write this note.

\section{Homology and cohomology of $\calB_d$}\label{s:h_ld}

The homology of the little discs operad is well known~\cite{Coh}:

\begin{theorem}[F. Cohen \cite{Coh}]\label{t:h_ld}
The homology operad $H_*\calB_d$  is the operad of associative unital algebras in case $d=1$ and the operad of graded unital Poisson algebras   with bracket of degree $(d-1)$ in case $d\geq 2$.
\end{theorem}

Below we briefly describe the geometrical meaning of this result. We would also like to suggest the expository paper~\cite{Sinha_ld}, where Cohen's theorem is explained in full detail.

In case $d=1$, the space $\calB_1(n)$, $n\geq 0$, has $n!$ contractible components. Thus its homology is concentrated in degree~0 and has rank $n!$. We get
$$
H_*\calB_1(n)=H_0\calB_1(n)=\Assoc(n).
$$
It is also obvious that the compositions agree.

In case $d\geq 2$, one has $\calB_d(0)=*$, and $\calB_d(2)\simeq S^{d-1}$. The generators of $H_0\calB_d(0)$, $H_0\calB_d(2)$ and of $H_{d-1}\calB_d(2)$ are respectively the elements $1$, $x_1\cdot x_2$ and $[x_1,x_2]$ of the Poisson operad.
Notice that the theorem above describes $H_*\calB_d(n)$ as a free $\Z$-module spanned by products of iterated brackets. The corresponding cycles are realized as products of spheres. For example,
 $[[x_1,x_2],x_3]\in H_*\calB_d(3)$ can be realized as $S^{d-1}\times S^{d-1}\to \calM_d^{(2)}(3)$, where the point 2 rotates around 1, and 3 rotates around 1 and 2. As another example $[x_1,x_2]\cdot [x_3,x_4]$ can be realized as $S^{d-1}\times S^{d-1}\to \calM_d^{(2)}(4)$, where 2 rotates around 1,
 and 4 does so around 3, moreover 1 and 2 stay far away from 3 and 4.

In Section~\ref{s:left_bim} we give a similar description of $H_*\calB_d^{(k)}(n)$ as a space spanned by products of iterated brackets with each such cycle realized by products of spheres.

\begin{theorem}[\cite{Arn,Coh}]\label{t:coh_conf}
The cohomology algebra $H^*\calB_d(n)$, $d\geq 2$, is generated by $\alpha_{ij}\in H^{d-1}\calB_d(n)$, $1\leq i\neq j\leq n$; the relations are $\alpha_{ij}=(-1)^d\alpha_{ji}$, $\alpha_{ij}^2=0$, $\alpha_{ij}\alpha_{jk}+\alpha_{jk}\alpha_{ki}+\alpha_{ki}\alpha_{ij}=0$.
\end{theorem}

To any monomial of this algebra one can assign a graph on vertices $1,\ldots,n$ by putting an edge between $i$ and $j$ for every factor $\alpha_{ij}$. It follows from the relations that a monomial is non-zero if and only if the corresponding graph is a forest.

In Section~\ref{s:coh_space} we will give a similar description of $H^*\calB_d^{(k)}(n)$ as a free $\Z$-module spanned by certain forests and quotiented out by natural relations. The product of such forests will essentially be their superposition similarly to the case of $H^*\calB_d(n)$.

\section{$H_*\calB_d^{(k)}$ as a left module and as a bimodule}\label{s:left_bim}

One has a natural inclusion
$$
\calB_d(n)\subset\calB_d^{(k)}(n),\,\, n\geq 0,
\eqno(\numb)\label{eq:inclusion}
$$
which is null homotopic for $k\geq 3$ (in case $k=2$ it is an identity map). This can be shown by pulling apart (one after another) the discs in the configuration. Such a path goes through disc configurations  with at most  double overlaps. (More generally any inclusion $\calB_d^{(k)}(n)\subset \calB_d^{(k+1)}(n)$ is null by a similar argument.)

\begin{definition}\label{d:mod_under}
We say that $M$ is a {\it left module} (respectively, {\it bimodule}) {\it under} an operad $\calO$ if it is a left module (respectively, bimodule) over $\calO$ and is endowed with a map of left modules (respectively, bimodules) $\calO\to M$.
\end{definition}

As example relevant to us, due to the inclusion~\eqref{eq:inclusion}, $\calB_d^{(k)}$ is a bimodule under $\calB_d$. In fact in one of the applications in Section~\ref{s:application} it will be important that $\calB_d^{(k)}$ is not only a bimodule, but also a bimodule {\it under} $\calB_d$.

$\Com$ will denote the operad of commutative unital algebras over $\Z$.

\begin{definition}
An operad $\calO$ in graded $\Z$-modules is called {\it augmented} if it is endowed with a surjective map of operads $\calO\to\Com$.
\end{definition}

Notice that all the operads $H_*\calB_d$, $d\geq 1$, are naturally augmented since they arise as the homology of topological operads. This in particular implies that $\Com$ is a bimodule under $H_*\calB_d$.

\begin{definition}
We say that $M$ is a {\it pointed} left module (respectively, bimodule) under an augmented operad $\calO$ if $M$ is a left module (respectively, bimodule) under $\calO$,  the structure map $\calO\to M$ factors through $\Com$, and moreover the map $\Com\to M$ is an inclusion.
\end{definition}

Since all the maps~\eqref{eq:inclusion} are null for any $k\geq 3$, the bimodules $H_*\calB_d^{(k)}$, $k\geq 3$, are pointed under $H_*\calB_d$.

One has a natural forgetful functor from the category of pointed left modules (respectively, pointed bimodules) to the category of symmetric sequences, which  has a left adjoint. For a given symmetric sequence this left adjoint functor produces a free pointed left module (respectively, bimodule) generated by this sequence. Notice that the obtained left module (respectively, bimodule) is not free in the usual sense since it contains $\Com$ on which the Lie part of $H_*\calB_d$ acts trivially.

\begin{theorem}\label{t:l_mod_main}
For $k\geq 3$, the pointed left module $H_*\calB_d^{(k)}$ under $H_*\calB_d$ is generated by a single element $\{x_1,\ldots,x_k\}\in H_{(k-1)d-1}\calB_d^{(k)}(k)$ which is symmetric or skew symmetric depending on the parity of $d$:
$$
\{x_{\sigma_1}\ldots x_{\sigma_k}\}=(-1)^{|\sigma|d}\{x_1\ldots x_k\},\,\, \sigma\in\Sigma_k.
\eqno(\numb)\label{eq:sym_rel}
$$
The only relation that the left action has is the \underline{generalized Jacobi}:
$$
\sum_{i=1}^{k+1}(-1)^{(i-1)d}\left[x_i,\{x_1,\ldots,\hat x_i,\ldots,x_{k+1}\}\right]=0.
\eqno(\numb)\label{eq:gen_jacob}
$$
\end{theorem}

The element $\{x_1,\ldots,x_k\}\in H_{(k-1)d-1}\calB_d^{(k)}(k)\simeq \Z$ can be realized by a sphere in $\calM_d^{(k)}(k)\simeq \calB_d^{(k)}(k)$:
$$
\sum_{i=1}^k x_i^2=\varepsilon^2;\qquad \sum_{i=1}^k x_i=0,
$$
where $x_i$ is the $i$-th point in the configuration (equivalently the center of the $i$-th disc). For the theorem above one can choose any orientation of this sphere. Orientation will matter only when we will be speaking about the duality between the homology and cohomology, see Section~\ref{s:duality}.

\begin{proof}[Proof of \eqref{eq:gen_jacob}] The generalized Jacobi is very easy to see. Consider the intersection of $\calM_d^{(k)}(k+1)$ with the $(kd-1)$-sphere $\sum_{i=1}^{k+1}x_i=0$, $\sum_{i=1}^{k+1}x_i^2=1$. The resulting space is homotopy equivalent to $\calM_d^{(k)}(k+1)$. This space is the sphere $S^{kd-1}$ from which one removed $(k+1)$ disjoint $(d-1)$-spheres. Now consider the $(kd-1)$-chain $C$ which is this sphere $S^{kd-1}$ minus small tubular neighborhoods of the aforementioned $(d-1)$-spheres. It is easy to see that the boundary of $C$ produces exactly relation~\eqref{eq:gen_jacob}.
\end{proof}

\begin{remark}\label{r:two_gener}
 As a pointed left module under $H_*\calB_d$,  the sequence $H_*\calB_d^{(k)}$, $k\geq 3$, is generated by a single element, but as a left module it is generated by two elements $x_1\in H_0 \calB_d^{(k)}(1)$ and $\{x_1\ldots x_k\}\in H_{(k-1)d-1}\calB_d^{(k)}(k)$. The left submodule generated by $x_1$ is exactly $\Com=H_0\calB_d^{(k)}$. The Lie part of $H_*\calB_d$ acts trivially on this submodule which is equivalent to the relation
 $$
 [x_1,x_2]=0.
 \eqno(\numb)\label{eq:lie_triv}
 $$
 Geometrically this relations says that rotating one disc around the other produces a trivial homology class in $\calB_d^{(k)}(2)\simeq *$, $k\geq 3$.
\end{remark}

Theorem~\ref{t:l_mod_main} tells us that the left action of $H_*\calB_d$ suffices to completely describe the homology groups $H_*\calB_d^{(k)}(n)$ as spaces spanned by products of iterated brackets. The right action of $H_*\calB_d$ on $H_*\calB_d^{(k)}$ tells us what happens with the homology when the points in configurations get multiplied -- this will be important for applications, see Section~\ref{s:application}.

\begin{theorem}\label{t:bimod_main}
 For $k\geq 3$, the pointed bimodule $H_*\calB_d^{(k)}$ under $H_*\calB_d$ is generated by a single element $\{x_1,\ldots,x_k\}\in H_{(k-1)d-1}\calB_d^{(k)}$ satisfying the symmetry~\eqref{eq:sym_rel}, generalized Jacobi~\eqref{eq:gen_jacob}, and Leibniz relations with respect to the right action:
 \begin{gather}
 \{x_1,\ldots,x_{k-1},x_k\cdot x_{k+1}\}= x_k\cdot \{x_1,\ldots,x_{k-1},x_{k+1}\}+
 \{x_1,\ldots,x_k\}\cdot x_{k+1}; \label{eq:leib_prod}\\
 \{x_1,\ldots,x_{k-1},[x_k, x_{k+1}]\}=(-1)^d\left[\{x_1,\ldots,x_{k-1},x_{k+1}\},x_k\right]+
 \left[\{x_1,\ldots,x_k\}, x_{k+1}\right].
 \label{eq:leib_bracket}
 \end{gather}
  \end{theorem}

    One can show that~\eqref{eq:leib_prod} implies
  $$
  \{x_1,\ldots,x_{k-1},1\}=0,
  $$
  where 1 is the generator of $\Com(0)=H_0B_d(0)$. Geometrically composition with this element forgets the corresponding disc in the configurations.

 Notice also that in the case $d=1$, the second relation~\eqref{eq:leib_bracket} follows from the first one~\eqref{eq:leib_prod}.

\begin{proof}[Proof of Theorem~\ref{t:bimod_main}] In order to prove this theorem it suffices to prove Theorem~\ref{t:l_mod_main} and also relations~\eqref{eq:leib_prod}, \eqref{eq:leib_bracket}. The latter relations are proved in Examples~\ref{ex:leib_prod_d1}, \ref{ex:leib_prod_d>1}, \ref{ex:leib_brack}.
\end{proof}


Theorem~\ref{t:l_mod_main} follows from Propositions~\ref{p:l_act_gener},~\ref{p:d1_basis},~\ref{p:d>1_basis}.

\begin{proposition}\label{p:l_act_gener}
The cycles obtained by the left action of $H_*\calB_d$ on $H_0\calB_d^{(k)}(1)$ and on $H_{(k-1)d-1}\calB_d^{(k)}(k)$ span the homology of each component $H_*\calB_d^{(k)}(n)$, $n\geq 0$.
\end{proposition}

The cases $d=1$ and $d\geq 2$ of this proposition are proved in Sections~\ref{s:d1} and~\ref{s:proof_l_act_gener} respectively. The case $d=1$ was  essentially done by Baryshnikov~\cite{Bar}. For $d>1$ the argument is an easy generalization of the case $d=1$. In order to complete the proof of Theorem~\ref{t:l_mod_main} one needs to show that between the cycles produced by this left action there is no other relations besides those that follow from~\eqref{eq:sym_rel},~\eqref{eq:gen_jacob},~\eqref{eq:lie_triv}. This is done by providing an explicit basis of $H_*\calB_d^{(k)}(n)$.

\begin{proposition}\label{p:d1_basis}
The homology $H_*\calB_1^{(k)}(n)$, $k\geq 3$, is torsion free. For its basis one can take the  set whose elements are encoded by partitions $I_0$, $J_1$, $I_1$, $J_2$,$\ldots$, $J_{\ell}$, $I_{\ell}$ of $\underline{n}$, such that $\ell\geq 0$, $| J_s|=k$, $s=1\ldots\ell$,  and $\max(I_s\sqcup J_{s+1})\in J_{s+1}$ for all $s=0,\ldots ,\ell-1$. The homology class corresponding to such partition has the form
$$
A_{I_0}\cdot B_{J_1}\cdot A_{I_1}\cdot B_{J_2}\cdot\ldots \cdot A_{I_{\ell-1}} \cdot B_{J_\ell} \cdot A_{I_{\ell}},
\eqno(\numb)\label{eq:d1_basis}
$$
where $A_{I_s}=x_{i_{1,s}}\cdot x_{i_{2,s}}\cdot\ldots x_{i_{|I_s|,s}}$, $I_s=\{i_{1,s}<i_{2,s}<\ldots<i_{|I_s|,s}\}$, (in case $I_s=\emptyset$, $A_{I_s}=1$ or is simply omitted);
$B_{J_s}=\{x_{j_{1,s}},x_{j_{2,s}},\ldots,x_{j_{k,s}}\}$, $J_s=\{j_{1,s}<j_{2,s}<\ldots<j_{k,s}\}$.
\end{proposition}

It follows from Proposition~\ref{p:l_act_gener} and relations~\eqref{eq:sym_rel},~\eqref{eq:gen_jacob},~\eqref{eq:lie_triv} that any homology class in $H_*\calB_1^{(k)}(n)$ is a linear combination of the elements~\eqref{eq:d1_basis}. In Section~\ref{s:d1} we will produce an explicit set of cohomology classes described by essentially the same combinatorial data such that the pairing matrix with~\eqref{eq:d1_basis} is upper triangular. This proves the linear independence of the elements~\eqref{eq:d1_basis}.

\begin{proposition}\label{p:d>1_basis}
The homology $H_*\calB_d^{(k)}(n)$, $d\geq 2$, $k\geq 3$, is torsion free. For its basis one can take the products of iterated brackets satisfying the following conditions: each factor is either $x_i$, $i\in\underline{n}$, or an iterated bracket of the form
$$
[\ldots[[B_1,B_2],B_3]\ldots B_\ell],\quad \ell\geq 1,
\eqno(\numb)\label{eq:iter_brack}
$$
where each $B_s$ is of the form
$$
B_s=[\ldots[[\{x_{j_{1,s}},x_{j_{2,s}},\ldots,x_{j_{k,s}}\},x_{i_{1,s}}],x_{i_{2,s}}]\ldots x_{i_{\ell_s,s}}],
$$
where $j_{1,s}<j_{2,s}<\ldots<j_{k,s}$; $\ell_s\geq 0$; $i_{1,s}<i_{2,s}<\ldots < i_{\ell_s,s}<j_{k,s}$. Also we require that the smallest index in~\eqref{eq:iter_brack} must appear  in $B_1$.
\end{proposition}

Again it follows from Proposition~\ref{p:l_act_gener} that the elements above span $H_*\calB_d^{(k)}(n)$. To prove that they are linearly independent we produce an explicit dual basis in cohomology, see Section~\ref{s:proof_l_act_gener}.

\begin{corollary}\label{c:wedge}
For any $d\geq 1$, $k\geq 2$, $n\geq 0$, the suspension $\Sigma \calM_d^{(k)}(n)$ is homotopy equivalent to a wedge of spheres.
\end{corollary}

\begin{proof}
This is always true if a space has torsion free homology admitting a basis realized by  products of spheres.
\end{proof}

\begin{remark}\label{r:case_k2}
In the case $k=2$, the homology $H_*\calB_d^{(2)}(\bullet)=H_*\calB_d(\bullet)$ admits a natural decreasing filtration that respects the structure of a bimodule over $H_*\calB_d$. The statements of Theorems~\ref{t:l_mod_main}, \ref{t:bimod_main}, and Propositions~\ref{p:l_act_gener}, \ref{p:d1_basis}, \ref{p:d>1_basis} hold if one replaces $H_*\calB_d^{(2)}$ by the associated graded quotient. This filtration was considered in~\cite{SundWachs}. As Sundaram and Wachs point out,  it is induced by  the Reutenauer derived series filtration in the free Lie algebra~\cite{Reut_char}.
\end{remark}

\section{Case $d=1$}\label{s:d1}

First we prove Proposition~\ref{p:l_act_gener} in case $d=1$. This was implicitly done by Baryshnikov in~\cite{Bar}. We repeat his argument for completeness of exposition. The space $\calB_1^{(k)}(0)$ is a point, $H_*\calB_1^{(k)}(0)\simeq\Z$ which is obtained by the left action of the arity zero component $H_*\calB_1(0)=H_0\calB_1(0)\simeq \Z$. The generator of $H_0\calB_1^{(k)}(0)\simeq H_0\calB_1(0)$ is denoted by~1. We then proceed by induction over $n$. Consider a cycle $[\alpha]\in H_*\calM_1^{(k)}(n)$ and a chain $\alpha$ representing $[\alpha]$.  Consider the projection $p\colon \calM_1^{(k)}(n)\to \calM_1^{(k)}(n-1)$ that forgets the last point in configurations. By a little perturbation one can assume that each simplex of $\alpha$ is smooth and transversal to every fiber of $p$. Define a homotopy $\alpha_t$, $0\leq t\leq c$, of  $\alpha$ in $\calM_1^{(k)}(n-1)\times \R$ by adding $t$ to the last coordinate $x_n$. (In other words we pull the last point $x_n$ to the right for every point in the cycle $\alpha$.) This homotopy viewed as a chain in $\calM_1^{(k)}(n-1)\times \R$ intersects transversely the forbidden fibers -- it happens when $x_n+t$ collides with $x_{i_1}=\ldots=x_{i_{k-1}}$, $1\leq i_1<i_2<\ldots <i_{k-1}\leq n-1$.  To turn $\alpha_t$ into a chain in $\calM_1^{(k)}(n)$ we remove from it intersections with small tubular neighborhoods of the planes
$x_{i_1}=\ldots=x_{i_{k-1}}=x_n$. One get that the boundary of such chain is the sum of $\alpha$ (when $t=0$), a cycle of the form $A\cdot x_n$, where $A\in H_*\calM_1^{(k)}(n-1)$ (when $t=c$), and cycles of
 the form $A_I\cdot\{x_{i_1},x_{i_2},\ldots,x_{i_{k-1}},x_n\}\cdot B_J$, where $A_I\in H_*\calM_1^{(k)}(I)$,
 $B_J\in H_*\calM_1^{(k)}(J)$, $I\sqcup J=\underline{n-1}\setminus \{i_1,i_2,\ldots,i_{k-1}\}$ (such cycles correspond to the part of the boundary appearing from the intersection of $\alpha_t$  with the plane
 $x_{i_1}=\ldots=x_{i_{k-1}}=x_n$).  The set $I$ (respectively $J$) contains the indices $i$ such that $x_i<x_n$ (respectively, $x_i>x_n$). Now using induction we get the result.\footnote{Notice that this recursive
  procedure shows that any cycle of $\calM_1^{(k)}(n)$ is homologous to a linear combination of the basis elements from Proposition~\ref{p:d1_basis}.}  Q.E.D.

\begin{example}\label{ex:leib_prod_d1}
Consider a natural chain representing  the cycle $\{x_1,x_2,\ldots,x_{k-1},x_k\cdot x_{k+1}\}\in H_{k-2}\calM_1^{(k)}(k+1)$. When $x_{k+1}$ is pulled to the right, it can only meet the plane $x_1=x_2=\ldots=x_{k-1}=x_{k+1}$, which produces the cycle $x_k\cdot\{x_1,\ldots,x_{k-1},x_{k+1}\}$. At the other end of the homotopy we get the cycle $\{x_1,\ldots,x_k\}\cdot x_{k+1}$. As a result we get exactly relation~\eqref{eq:leib_prod}.
\end{example}

Now we prove Proposition~\ref{p:d1_basis}. We will exhibit an explicit dual basis in  cohomology.  We reiterate that  it was done in~\cite{Bar} and we give it for completeness of exposition.

For a partition of $\underline{n}$ into a collection of subsets $I_0$, $J_1$, $I_1$, $J_2$,$\ldots$, $I_{\ell-1}$, $J_\ell$, $I_{\ell}$, define a subset of points in $\R^n$ satisfying the following (in)equalities:
\begin{eqnarray}
x_i\leq x_j, & i\in I_s, \,\, j\in J_{s+1};\\
x_j\leq x_i,& j\in J_s,\,\, i\in I_{s};\\
x_{j_1}=x_{j_2},& j_1,j_2\in J_s.
\end{eqnarray}\label{eq:inequalities}
This set or rather its intersection with $\calM_1^{(k)}(n)$ will be denoted by
$$
(I_0)[J_1](I_1)[J_2]\ldots (I_{\ell-1})[J_\ell](I_{\ell}).
\eqno(\numb)\label{eq:d1_chains}
$$
Now let us assume that $|J_s|=k-1$ for all $s=1\ldots\ell$. We get that the boundary of this set (viewed as a locally compact chain) lies in the complement of $\calM_1^{(k)}(n)$. Thus via intersection number it defines a cocycle in $H^*\calM_1^{(k)}(n)$.  In addition assuming the restriction
$$
\max(I_s\sqcup J_{s+1})\in I_s
\eqno(\numb)\label{eq:max_cochain}
$$
we get a collection of cocycles which is exactly a basis dual to~\eqref{eq:d1_basis}.\footnote{To be precise  for an appropriate order of elements the pairing  is given by an upper triangular matrix with $\pm 1$ on the diagonal. We leave it as an exercise to the reader.}  Without the second restriction~\eqref{eq:max_cochain}  (but still assuming $|J_s|=k-1$ for all $s=1\ldots\ell$)  the cocycles~\eqref{eq:d1_chains} are linearly dependent in $H^*\calM_1^{(k)}(n)$.  Baryshnikov shows that all relations are spanned by boundaries of the chains~\eqref{eq:d1_chains}  with all $J_s$ of cardinal $k-1$ except one of cardinal  $k-2$.  Moreover Baryshnikov describes the cohomology algebra $H^*\calM_1^{(k)}(n)$ as being generated by the elements $(I_0)[J_1](I_1)$, $|J_1|=k-1$. The relations are linear appearing as boundary of the elements $(I_0\rq{})[J_1\rq{}](I_1\rq{})$, $|J_1\rq{}|=k-2$; and quadratic: the square of any element $(I_0)[J_1](I_1)$ is zero; and the product of two generators is zero if the intersection of the corresponding locally finite chains in $\calM_1^{(k)}(n)$ is empty.

\section{Proof of Proposition~\ref{p:l_act_gener} for $d\geq 2$}\label{s:proof_l_act_gener}
The proof of Proposition~\ref{p:l_act_gener} for $d\geq 2$ is similar to the case $d=1$. Given a cycle in $\calM_d^{(k)}(n)$ we will homotop it by pulling the last point $x_n$ in the configuration away from the other points. This will lead to a similar recursive construction, but the recursion will be using the homology of slightly more general arrangements. Denote by $\calM_d^{(k)}(n,m)$ the complement in $\R^{d(n+m)}=
\left\{(x_1,\ldots,x_n;y_1,\ldots, y_m)\, |\, x_i\in \R^d,\, y_j\in \R^d\right\}$ to the union of subspaces
$$
x_{i_1}=\ldots=x_{i_k},
$$
for any cardinal $k$ subset $\{i_1,\ldots,i_k\}\subset \underline{n}$,
\begin{eqnarray*}
x_i=y_j,& 1\leq i\leq n,\quad 1\leq j\leq m;\\
y_{j_1}=y_{j_2},& 1\leq j_1< j_2\leq m.
\end{eqnarray*}

The space $\calM_d^{(k)}(n,m)$ is homotopy equivalent to the space $\calB_d^{(k)}(n,m)$ of configurations of $n$ discs labeled by $1,2,\ldots,n$ and colored by $x$, and of $m$ discs labeled by $1,2,\ldots,m$ and colored by $y$, in a unit disc. The {\it non-overlapping condition} is that no $k$ $x$-colored discs have a non-trivial intersection, and all $y$-colored discs are disjoint one from another and from the $x$-discs.

We say that a family of spaces (or vector spaces) $M(n,m)$, $n\geq 0$, $m\geq 0$, is a {\it bi-colored left module}  over an operad $\calO$ if each $M(n,m)$ is acted on by $\Sigma_n\times \Sigma_m$, and one is given structure composition maps:
$$
\calO(\ell)\times M(n_1,m_1)\times M(n_2,m_2) \times\ldots \times M(n_\ell,m_\ell)\to M(n_1+\ldots+n_\ell,m_1+\ldots+m_\ell).
\eqno(\numb)\label{eq:color_compos}
$$
One assumes the easily guessed symmetric group equivariance, associativity, and unity conditions. As example $\calB_d^{(k)}(\bullet,\bullet)$ is a bi-colored left module over $\calB_d(\bullet)$. A similar structure is induced in homology.

\begin{theorem}\label{t:col_lef_main}
For $d\geq 2$, $k\geq 3$, the bi-colored left module $H_*\calB_d^{(k)}(\bullet,\bullet)$ is generated by  $x_1\in H_0 \calB_d^{(k)}(1,0)$, $\{x_1\ldots x_k\}\in H_{(k-1)d-1}\calB_d^{(k)}(k,0)$, and $y_1\in H_0\calB_d^{(k)}(0,1)$. The only relations are \eqref{eq:sym_rel}, \eqref{eq:gen_jacob}, \eqref{eq:lie_triv}.
\end{theorem}

The theorem above describes the homology of each component $\calB_d^{(k)}(n,m)$ as a space spanned by products of iterated brackets on $x_1,\ldots,x_n,y_1,\ldots,y_m$. The proof of this theorem is very similar to that of Theorem~\ref{t:l_mod_main}. We will only show that the elements obtained by the left action of $H_*\calB_d$ on $x_1$, $y_1$, and $\{x_1\ldots x_k\}$ do span the homology of each component $H_*\calB_d^{(k)}(n,m)$. This will obviously imply Proposition~\ref{p:l_act_gener}.

For $n=0$ the statement is obvious. Indeed, $H_*\calB_d^{(k)}(0,\bullet)$ is isomorphic to $H_*\calB_d(\bullet)$ as a left $H_*\calB_d$-module: it is freely generated by the single element $y_1\in H_0\calB_d^{(k)}(0,1)$.  Now let $\alpha$ be a smooth generic $s$-dimensional chain (by this we mean each simplex is smooth and in generic position) in $\calM_d^{(k)}(n,m)$. We consider the homotopy $\alpha_t$, $0\leq t\leq c$, of $\alpha$ in $\calM_d^{(k)}(n-1,m)\times \R^d$ that only affects the last coordinate $x_n(t)=x_n+t\cdot v$, where the vector $v\in \R^d\setminus \{0\}$ is fixed. When $c$ is big enough $x_n(c)$ will be far away from all the other points $x_1,\ldots,x_{n-1},y_1,\ldots,y_m$ appearing in $\alpha$. The fact that $\alpha$ is generic and smooth garantees  that $\alpha_t$ viewed as an $(s+1)$-chain in $\calM_d^{(k)}(n-1,m)\times \R^d$ is transversal to the forbidden subspaces
$$
x_n=y_j, \quad 1\leq j\leq m;
\eqno(\numb)\label{eq:x=y}
$$
$$
x_n=x_{i_1}=x_{i_2}=\ldots=x_{i_{k-1}},\quad 1\leq i_1< i_2<\ldots< i_{k-1}\leq n-1.
\eqno(\numb)\label{eq:x=x=x}
$$
We remove from $\alpha_t$ intersections with small tubular neighborhoods of the above subspaces. The boundary of the obtained chain is the initial cycle $\alpha$ (when $t=0$), a cycle of the form $A\cdot x_n$, where $A\in H_*\calM_d^{(k)}(n-1,m)$ (this cycle appears at the other end of the homotopy $t=c$), the cycles of the form $A|_{y_j=[y_j,x_n]}$, where $A\in H_*\calM_d^{(k)}(n-1,m)$ (such cycles appear from intersection of $\alpha_t$ with~\eqref{eq:x=y}), and the cycles of the form $A|_{y_{m+1}=\{x_{i_1},\ldots,x_{i_{k-2}},x_{i_{k-1}},x_n\}}$, where $A\in H_*\calM_d^{(k)}(\underline{n-1}\setminus I,m+1)$ and $I=\{i_1,\ldots,i_{k-1}\}$ (such cycles appear from intersection of $\alpha_t$ with~\eqref{eq:x=x=x}).  Using induction hypothesis we   express $\alpha$ as a linear combination of products of iterated brackets. Q.E.D.

\begin{example}\label{ex:leib_prod_d>1}
Consider the cycle $\{x_1,\ldots,x_{k-1},x_k\cdot x_{k+1}\}\in H_{(k-1)d-1}\calM_d^{(k)}(n+1,0)$. While pulling away $x_{k+1}$ one can only meet the plane
$$
x_{k+1}=x_1=x_2=\ldots=x_{k-1},
$$
which produces the cycle $x_k\cdot \{x_1,\ldots,x_{k-1}, x_{k+1}\}$. At the second end we get the cycle $\{x_1,x_2,\ldots,x_k\}\cdot x_{k+1}$. This proves relation~\eqref{eq:leib_prod}.
\end{example}

\begin{example}\label{ex:leib_brack}
Now let us apply the above procedure to the cycle $\{x_1,\ldots,x_{k-1},[x_k, x_{k+1}]\}\in H_{kd-2}\calM_d^{(k)}(n+1,0)$.  While pulling $x_{k+1}$ away one meets the planes
$$
x_{k+1}=x_1=\ldots=\widehat{x_i}=\ldots=x_k,
$$
$i=1\ldots k-1$, which produces the cycles
$$
\left[\{x_1\ldots \widehat{x_i}\ldots x_{k+1}\},x_i\right];
$$
the plane
$$
x_{k+1}=x_1=\ldots=x_{k-1},
$$
which produces the cycle
$$
[x_k,1]\cdot\{x_1,\ldots,x_{k-1},x_{k+1}\}=0.
$$
Also at the other end of the homotopy we get the cycle
$$
\{x_1,\ldots,x_{k-1},[x_k,1]\}\cdot x_{k+1}=0.
$$
As a result we get
$$
\left\{x_1,\ldots,x_{k-1},[x_k, x_{k+1}]\right\}=-\sum_{i=1}^{k-1} (-1)^{(k+1-i)d} \left[\{x_1\ldots \widehat{x_i}\ldots x_{k+1}\},x_i\right].
$$
Applying the generalized Jacobi identity~\eqref{eq:gen_jacob} we get~\eqref{eq:leib_bracket}.
\end{example}

\begin{remark}\label{r:poset_homol}
In the initial work~\cite{BjorWelk} the (co)homology of the poset $\Pi_{n,k}$ was computed recursively by introducing  auxiliary lattices $\Pi_{n,k}(\ell)$. The argument of this section gives a geometric explanation for this combinatorial recursion.
\end{remark}

\section{Cohomology $H^*\calB_d^{(k)}(n)$ as a space of forests}\label{s:coh_space}
Recall that the cohomology of $\calB_d(n)\simeq\calM_d^{(2)}(n)$ is described as a certain space of forests modulo 3-terms relations, see Section~\ref{s:h_ld}. In this section we will give a similar description of
$H^*\calB_d^{(k)}(n)=H^*\calM_d^{(k)}(n)$, $k\geq 3$, as spaces of certain {\it admissible $k$-forests} modulo narural relations. The $k$-forests that span $H^*\calM_d^{(k)}(n)$ have 2 types of vertices: square ones that contain cardinality $(k-1)$ subsets of $\underline{n}$, and round ones that contain only one element  from $\underline{n}$.  Every round vertex must be either disconnected from all the other vertices or connected to a single one that must be square. Every square vertex must be connected to at least one  round one.
Every element from $\underline{n}$ must appear in exactly one vertex of such $k$-forest. By an {\it orientation} of a $k$-forest we will understand
\begin{enumerate}[(a)]
\item orientation of each edge;\label{item:a}
\item ordering elements inside each square vertex;\label{item:b}
\item ordering {\it orientation set} consisting of all the edges (considered as elements of degree $d-1$) and all the square vertices (considered as elements of degree $k(d-2)$) in the $k$-forest.
\label{item:c}
\end{enumerate}

\begin{center}
\begin{figure}[h]
\includegraphics[width=10cm]{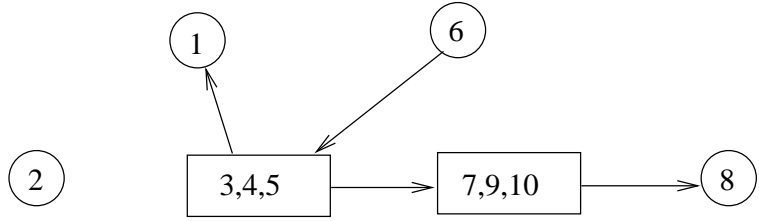}
\caption{Example of an admissible $4$-forest. This forest represents an element in $H^*\calM^{(4)}_d(10)$.}\label{fig1}
\end{figure}
\end{center}

For every such oriented forest $T$, we will assign a locally compact cooriented chain in $\calM_d^{(k)}(n)$, whose boundary lies in the complement of  $\calM_d^{(k)}(n)$. Thus every such chain defines a cocycle in $H^*\calB_d^{(k)}(n)$ (abusing notation it will be also denoted by~$T$), whose degree~$|T|$ is the sum of degrees of the elements in the orientation set.

By $p_1\colon\R^d\to\R^{d-1}$ we will denote the projection $(x^1,\ldots,x^d)\mapsto (x^2,\ldots,x^d)$.
The chain corresponding to a $k$-forest $T$ is defined as a set determined by the following (in)equalities:
\begin{itemize}
\item If $i$ and $j$ from $\underline{n}$ lie in the same square vertex, then $x_i=x_j$;
\item If two vertices $A$ and $B$ of $T$ are connected by an edge oriented from $A$ to $B$, then for all $i\in A$, $j\in B$, one has $x_i^1\leq x_j^1$ and $p_1(x_i)=p_1(x_j)$.
\end{itemize}

Notice that in particular if $i$ and $j$ from $\underline{n}$ lie in the same connected component of $T$, then $p_1(x_i)=p_1(x_j)$. The data \eqref{item:b}, \eqref{item:c} of the orientation of $T$ determine
 the coorientation of this chain. Notice that each chain is a convex domain of a vector subspace of codimension $|T|$ in $\R^{nd}$.
 The coorientation will be given by an explicit map $\R^{nd}\to\R^{|T|}$, where
 $\R^{|T|}$ is the product of $\R^{d-1}$\rq{}s (one copy for each edge) and of $\R^{(k-2)d}$\rq{}s (one
 copy for each square vertex) appearing in the same order as the corresponding elements appear in the
 orientation set of $T$. Given an edge from a vertex $A$ to $B$, we take the first elements $i\in A$ and $j\in B$ (to recall each such set is ordered being either singleton or by the oriention data~\eqref{item:b}).
  The projection $p_{AB}\colon \R^{nd}\to\R^{d-1}$ corresponding to this edge sends
$$
(x_1,\ldots,x_n)\mapsto p_1(x_j-x_i).
$$
 Given a square vertex $A$, whose ordered set of elements is $(i_1,i_2,\ldots,i_{k-1})$, the corresponding projection $p_A\colon \R^{nd}\to\R^{(k-2)d}$ sends
$$
(x_1,\ldots,x_n)\mapsto (x_{i_2}-x_{i_1},x_{i_3}-x_{i_1},\ldots,x_{i_k}-x_{i_1}).
$$

\renewcommand{\labelenumi}{\arabic{enumi}.}
\renewcommand{\labelenumii}{\arabic{enumi}.\arabic{enumii}}

\begin{theorem}\label{t:coh_space}
The cohomology $H^*\calB_d^{(k)}(n)=H^*\calM_d^{(k)}(n)$, $d\geq 2$, $k\geq 3$, $n\geq 0$, has no torsion and can be described as a space spanned by oriented $k$-forests on the index set $\underline{n}$ and quotiented out by the following relations:
 \begin{enumerate}
 \item Orientation relations:
   \begin{enumerate}
   \item Changing the order of the orientation set produces the Koszul sign of permutation;
   \item A permutation $\sigma\in\Sigma_{k-1}$ of elements inside a square vertex produces the sign~$(-1)^{|\sigma|d}$.
   \item Changing orientation of an edge produces the sign $(-1)^d$;
   \end{enumerate}
 \item 3-term relations:
 \begin{center}
\includegraphics[width=13cm]{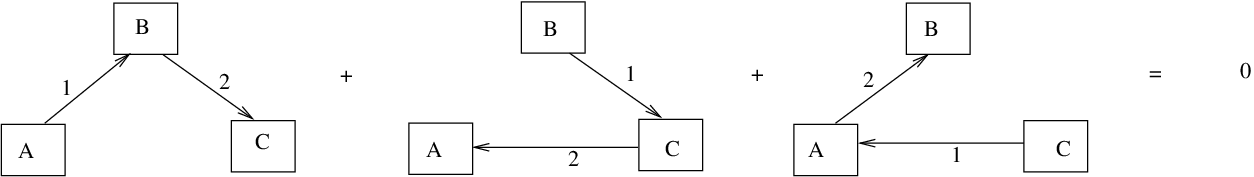}
\end{center}
(This picture is local --- we assume that the three forests are identical except for the edges going between the square vertices $A$, $B$, $C$. The numbers on the edges tell in which order the edges appear in the orientation set.)
 \item Relations  dual to the generalized Jacobi:
 $$
 \sum_{\ell=1}^m (-1)^{\ell(d-1)}
 \raisebox{-1.5cm}{
 \psfrag{i1jl}[0][0][1][0]{$i_1i_2\ldots i_{k-2}j_\ell$}
 \psfrag{j1}[0][0][1][0]{$j_1$}
 \psfrag{j2}[0][0][1][0]{$j_2$}
 \psfrag{jl-1}[0][0][1][0]{$j_{\ell-1}$}
 \psfrag{jl+1}[0][0][1][0]{$j_{\ell+1}$}
 \psfrag{jm}[0][0][1][0]{$j_m$}
 \includegraphics[width=7cm]{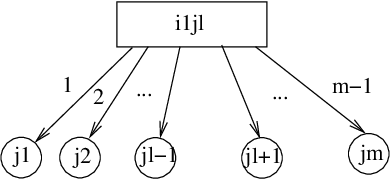}
 }
 =0
 \eqno(\numb)\label{eq:dual_jacobi}
 $$
 (Again this picture is local. The square vertex above may be connected to other square vertices, but not to round ones.)
 \end{enumerate}
\end{theorem}

 \begin{proof}
 First let us check that the cocycles corresponding to $k$-forests satisfy all the relations above. Relations (1.1) and (1.2) appear as a change of coorientation. To see (1.3) one notices that changing orientation of an edge produces a different chain: instead of inequality $x_i^1\leq x_j^1$ one would have $x_i^1\geq x_j^1$. Up to a sign $(-1)$ these two chains are homologous. Also their coorientation differs by $(-1)^{d-1}$. Thus the total sign contribution is $(-1)\cdot (-1)^{d-1}=(-1)^d$.

 Relation (2) is equivalent to
 \begin{center}
\includegraphics[width=12cm]{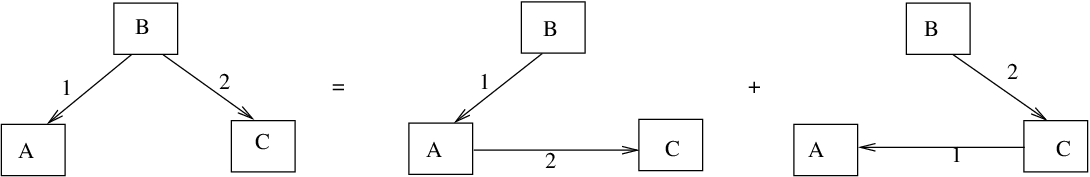}
\end{center}

 But the chain representing the left-hand side is exactly the union of the chains from the right-hand side.

   Relation (3) appears as the boundary of a similar chain that can be described by a similar forest one of whose square vertices has $k-2$ elements:
\begin{center}
\psfrag{i1jl}[0][0][1][0]{$i_1i_2\ldots i_{k-2}$}
 \psfrag{j1}[0][0][1][0]{$j_1$}
 \psfrag{jm}[0][0][1][0]{$j_m$}
 \includegraphics[width=4cm]{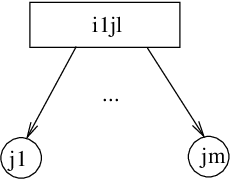}
\end{center}

 \begin{remark}\label{r:sq_vert_no_round}
 Relation (3) makes sense for $m=1$. In other words if we allow $k$-forests with square vertices not-attached to any round vertex, then the corresponding cocycles are zero in cohomology. This will be important in the next section where we will be studying the multiplicative structure of $H^*\calM_d^{(k)}(n)$.
 \end{remark}

 To finish the proof of Theorem~\ref{t:coh_space} we have to show that our $k$-forests cocycles span the entire cohomology and that there is no other relations. We will prove it by providing an explicit basis (in the space of such forests) that will be dual to the basis in homology described by Proposition~\ref{p:d>1_basis}.
 The fact that the intersection pairing is given by an identity matrix will finish the proof of Proposition~\ref{p:d>1_basis} as well. Our basis elements will be forests whose all components are either singletons or {\it linear $k$-trees}:
\begin{center}
\psfrag{A1}[0][0][1][0]{$A_1$}
 \psfrag{A2}[0][0][1][0]{$A_2$}
 \psfrag{A3}[0][0][1][0]{$A_3$}
 \psfrag{As}[0][0][1][0]{$A_s$}
 \includegraphics[width=8cm]{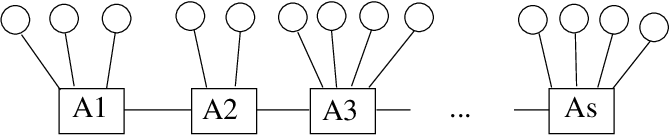}
\end{center}

 For a component $T_0$ as above we will require the following: the elements inside each square vertex appear in their natural linear order. The round vertices attached to every square vertex  also appear in their linear order. The last round vertex attached to $A_i$ is greater than the last element inside $A_i$. The minimal element in $T_0$ appears either inside $A_1$ or as a round vertex attached to $A_1$.

 We leave it as an exercise to the reader that the intersection matrix between the locally finite cycles corresponding to the aforementioned collection of $k$-forests and the cycles from Proposition~\ref{p:d>1_basis} is identity.  Otherwise the reader might wait until Section~\ref{s:duality} where the duality between the homology and cohomology is described more explicitly.

 \end{proof}

 \section{Multiplicative structure in cohomology}\label{s:multiplicative}
 In the previous section we described $H^*\calB_d^{(k)}(n)=H^*\calM_d^{(k)}(n)$ as a space spanned by certain $k$-forests. We will now describe the product which is essentially given by a superposition of such forests. The theorem below makes this statement more precise.

 \begin{theorem}\label{t:product}
 The product of two $k$-forest cocycles $T_1,T_2\in H^*\calM_d^{(k)}(n)$ is zero in cases (1)-(3) below. Otherwise it is a sum of $k$-forests as defined by (4)-(5).
 \begin{enumerate}[(1)]
 \item If there exists at least one square vertex $A$ in $T_1$ and one square vertex $B$ in $T_2$ such that $A\cap B\neq \emptyset$, then $T_1\cdot T_2=0$.

     \noindent (In (2)-(5) below we are assuming that all square vertices of $T_1$ are disjoint from those of $T_2$. In such situation one can define a superposition of two forests denoted by $T_1\cup T_2$.)

 \item In case $T_1\cup T_2$ has cycles then $T_1\cdot T_2=0$.

 \item In case $T_1\cup T_2$ has a square vertex without any round vertex attached then $T_1\cdot T_2=0$.

 \item If  $T_1\cup T_2$ is an admissible $k$-forests then $T_1\cdot T_2=T_1\cup T_2$, whose orientation set is obtained by concatenation of two orientation sets.

 \item It might happen that $T_1\cup T_2$ has one or several round vertices of valence 2. In such case one has to use the 3-term relations as follows in order to write $T_1\cdot T_2$ as a sum of admissible $k$-forests:
     \begin{center}
\includegraphics[width=12cm]{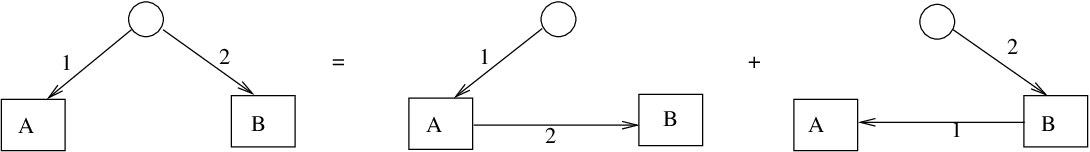}
\end{center}

 \end{enumerate}
 \end{theorem}

 \begin{proof}
 (1) In case $A\neq B$, the intersection of the chains corresponding to $T_1$ and $T_2$ is empty in $\calM_d^{(k)}(n)$. In case $A=B$, one can slightly move one of the chains to get an empty intersection.

 (2) One can choose an orientation of edges in $T_1$ and $T_2$ so that the intersection of the corresponding chains is empty.

 (3) See Remark~\ref{r:sq_vert_no_round}.

 (4) The chain corresponding to $T_1$ and $T_2$ are transversal one to another and their intersection is exactly the chain corresponding  to $T_1\cup T_2$.

  (5) Same as proof of relation~2 in Theorem~\ref{t:coh_space}.
\end{proof}

 The theorem below describes $H^*\calB_d^{(k)}(n)$ as a  quadratic algebra. For a pair of vertices $A$ and $B$ of a $k$-forest joined by an edge, we will agree to denote this edge either by $(A,B)$ or by $(i,j)$, where $i$ is any element in $A$, and $j$ any element in $B$.

 \begin{theorem}\label{t:quadratic}
 The algebra $H^*\calM_d^{(k)}(n)$ is generated by the forests that have only one square vertex. (Therefore only one of their components is not a singleton.) The relations are as follows:
 \begin{enumerate}[(1)]
 \item Linear relations in the space of generators as described by (1) and (3) from Theorem~\ref{t:coh_space}.

 \item $T_1\cdot T_2=0$ if the square vertex of $T_1$ is not disjoint from that of $T_2$ (in particular $(T_1)^2=0$).

 \item $T_1\cdot T_2=0$ if $T_1\cup T_2$ has cycles.

 \item $T_1\cdot T_2=0$ if $T_1\cup T_2$ has a square vertex without any round vertex attached. (This can happen if the square vertex of one of the forests has only one round vertex attached and which belongs to the square vertex of the second forest.)

 \item Let $i$ belong to a square vertex of $T_1$ and $j$ belong to a square vertex of $T_2$ also assume that the edge $a=(i,j)$ belongs to $T_1$. Then one has a relation $T_1\cdot T_2=(T_1\setminus a)\cdot (T_2\cup a)$, where $T_1\setminus a$ is a forest obtained from $T_1$ by removing $a$, and $T_2\cup a$ is obtained from $T_2$ by adding $a$:
\begin{center}
\includegraphics[width=6cm]{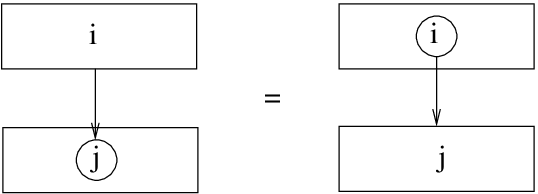}
\end{center}
 The sign is positive assuming that the edge $a$ is the last element in the orientation set of $T_1$ and the first element in the orientation set of $T_2\cup a$.

 \item If $T_1\cup T_2$ happens to have a bivalent round vertex, one gets a quadratic relation that one can draw as follows:
\begin{center}
\includegraphics[width=12cm]{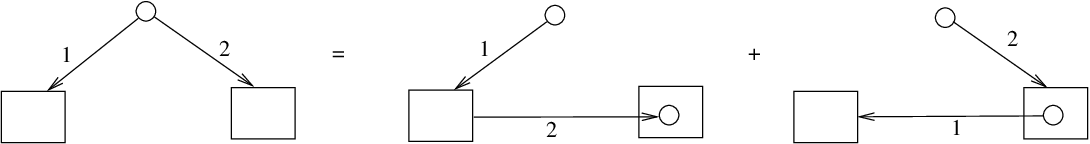}
\end{center}
\end{enumerate}
\end{theorem}

\begin{proof}
It is clear that any admissible $k$-forest can be obtained as a product of generators as above. It is also straightforward that relations from Theorem~\ref{t:coh_space} follow from relations  above and vice versa.
\end{proof}

\begin{remark}\label{r:case_k2_coh}
For the case $k=2$, Theorems~\ref{t:coh_space}, \ref{t:product}, \ref{t:quadratic} and also the duality between the homology and cohomology described in the next section, still hold if one replaces $H^*\calM_d^{(2)}(n)$ with a natural associated graded quotient, see Remark~\ref{r:case_k2}.
\end{remark}

\section{Duality between homology and cohomology}\label{s:duality}
So far we described the homology $H_*\calM_d^{(k)}(n)$, $d\geq 2$, $k\geq 3$, as a certain space spanned by products of iterated brackets, where each such product of brackets is a cycle realized by products of spheres in $\calM_d^{(k)}(n)$. We also described the cohomology  $H^*\calM_d^{(k)}(n)$, $d\geq 2$, $k\geq 3$, as a space spanned by admissible $k$-forests, where each forest is a cocycle realized via intersection number with certain locally finite chain. In this section we will describe how the aforementioned cycles pair with the cocycles or in other words how the cycles (realized by products of spheres) intersect with the locally finite chains described in Section~\ref{s:coh_space}. A similar duality for $\calM_d^{(2)}(n)$ is well known~\cite{Tur_other,Sinha_pair}. Notice that in top degree $H_*\calB_d(\bullet)$ is the operad of graded Lie algebras with bracket of degree $(d-1)$. Thus $H^*\calB_d(\bullet)$ in top degree is the Lie cooperad whose components are explicitly described as spaces of trees quotiented out by 3-term relations, see Section~\ref{s:h_ld}.  Such description of the Lie cooperad is important in its application to the rational homotopy theory~\cite{SinhaWalt1,SinhaWalt2}. Also it was used to prove the formality of the operad of little discs~\cite{Kontsevich,LV}.

Let $\calF_d^{(k)}(n)$ denote the space of admissible $k$-forests from Theorem~\ref{t:coh_space} modulo  only orientation relations~(1).  Then  $H^*\calM_d^{(k)}(n)$ is $\calF_d^{(k)}(n)$ quotiented out by the subspace $\calR_d^{(k)}(n)\subset \calF_d^{(k)}(n)$ spanned by relations (2) and (3):
$$
H^*\calM_d^{(k)}(n)=\left. \calF_d^{(k)}(n)\right/ \calR_d^{(k)}(n).
$$
The space $\calF_d^{(k)}(n)$ is naturally self-dual by defining  its basis set (of admissible $k$-forests) to be orthonormal. The homology $H_*\calM_d^{(k)}(n)$ is dual to $H^*\calM_d^{(k)}(n)$ and can be described as the subspace $\left(\calR_d^{(k)}(n)\right)^\bot\subset\calF_d^{(k)}(n)$. We will describe explicitly this isomorphism
$$
\Psi_n\colon H_*\calM_d^{(k)}(n)\to \left(\calR_d^{(k)}(n)\right)^\bot,
$$
which in fact encodes the pairing as
$$
\Psi_n(B)=\sum_T\langle T,B\rangle\cdot T.
$$
Here the sum is taken over the basis set of $\calF_d^{(k)}(n)$.

For simplicity of notation we will be omitting the subscript $n$. This map $\Psi$ can be described recursively. First we define $\Psi(1)$ as the empty graph and
$$
\Psi(x_i)=\raisebox{-.2cm}{\includegraphics[width=.6cm]{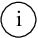}}
,
$$
where the right-hand side is the forest with only one vertex. We also define
$$
\Psi(\{x_{i_1}\ldots x_{i_k}\})=\sum_{\ell=1}^k (-1)^{(\ell-1)d}
\raisebox{-1cm}{
\psfrag{i1jl}[0][0][1][0]{$i_1\ldots \widehat{i_\ell} \ldots i_k$}
 \psfrag{il}[0][0][1][0]{$i_\ell$}
 \includegraphics[width=2.5cm]{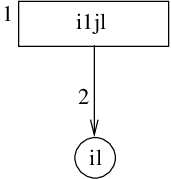}
 }
.
$$
The numbers 1 and 2 above describe the order in which the corresponding elements appear in the orientation set. This identity means that the spherical cycle $\{x_{i_1}\ldots x_{i_k}\}\in\calM_d^{(k)}(\{i_1\ldots i_k\})$ intersects each chain \raisebox{-1cm}{
\psfrag{i1jl}[0][0][1][0]{$i_1\ldots \widehat{i_\ell} \ldots i_k$}
 \psfrag{il}[0][0][1][0]{$i_\ell$}
 \includegraphics[width=2.2cm]{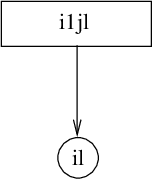}
 } exactly once, and $(-1)^{(\ell-1)d}$ is the sign of intersection.\footnote{At this point we need to fix orientation of the sphere $\{x_1\ldots x_k\}\in H_{(k-1)d-1}\calM_d^{(k)}(n)$ in order to make this pairing work.} Then if $B$ happens to be a product $B=B_1\cdot B_2$, we get
$$
\Psi(B_1\cdot B_2):=\Psi(B_1)\sqcup\Psi(B_2).
$$
If $B=[B_1,B_2]$ and neither $B_1$ nor $B_2$ is a singleton we get
$$
\Psi([B_1,B_2])=\sum_{{A_1\in B_1^\square}\atop  {A_2\in B_2^\square}}
\Psi(B_1)\cup (A_1,A_2)\cup \Psi(B_2),
\eqno(\numb)\label{eq:pairing_brack}
$$
where $B_1^\square$ (respectively $B_2^\square$) is the set of square vertices of each summand of $\Psi(B_1)$  (respectively of $\Psi(B_2)$)\footnote{Notice that the set of square vertices  for each summand of $\Psi(B_1)$ (respectively $\Psi(B_2)$)  is in one-to-one correspondence with the long brackets in $B_1$ (respectively $B_2$).}; $(A_1,A_2)$ is the edge going from $A_1$ to $A_2$. The orientation set for each summand is obtained by writing first the orientation set of a summand of $\Psi(B_1)$, then $(A_1,A_2)$, then the orientation set for a summand of $\Psi(B_2)$.

One similarly has
$$
\Psi([B,x_i])=\sum_{A\in B^\square}\Psi(B)\cup (A,i)\cup \Psi(x_i).
\eqno(\numb)\label{eq:pairing_singl}
$$

\begin{example}\label{ex:psi}

\begin{enumerate}[(a)]
\item
$
\Psi\left(\left[\{x_1\ldots x_k\},x_{k+1}\right]\right)=
\sum_{\ell=1}^k (-1)^{(\ell-1)d}
\raisebox{-1cm}{
\psfrag{1lk}[0][0][1][0]{$1\ldots \widehat{\ell} \ldots k$}
 \psfrag{l}[0][0][1][0]{$\ell$}
 \psfrag{k+1}[0][0][1][0]{$k{+}1$}
 \includegraphics[width=2.5cm]{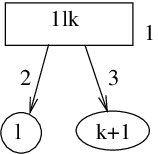}
 }.
$
\item
$
\psi\left(\left[\{x_1\ldots x_k\},\{x_{k+1}\ldots x_{2k}\}\right]\right)=$
$$
\sum_{i,j=1}^k(-1)^{(i+j)d}
\raisebox{-1cm}{
\psfrag{1ik}[0][0][1][0]{$1\ldots \widehat{i} \ldots k$}
\psfrag{k+1_2k}[0][0][1][0]{$k{+}1\ldots \widehat{k{+}j} \ldots 2k$}
 \psfrag{i}[0][0][1][0]{$i$}
 \psfrag{k+j}[0][0][1][0]{$k{+}j$}
 \includegraphics[width=7cm]{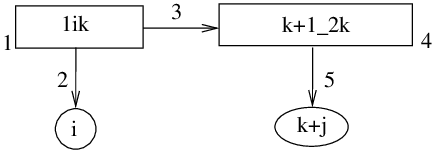}
 }.
$$
\end{enumerate}
\end{example}

\begin{remark}\label{r:gen_forest}
One can consider a slightly larger class of admissible $k$-forests by allowing round vertices to be connected to any number of square vertices. The advantage of such definition is that the multiplicative structure will be given simply by the superposition of  forests. The downside is that the space of cohomology would be less clearly described. But anyway if one decides to do so one will also need to take into account in the formula for pairing the intersections with the new locally finite chains. In the latter case the formula for~\eqref{eq:pairing_brack} and~\eqref{eq:pairing_singl} will be the same -- the sum will run over all vertices $A_1$ in $\Psi(B_1)$ and $A_2$ in $\Psi(B_2)$ with the only restriction that at least one of the two is square.
\end{remark}

\section{Coproduct and cobimodule structures}\label{s:co}
\subsection{Coproduct}\label{ss:coproduct}
Since $\calB_d$ is a topological operad, its homology is an operad in coalgebras. This structure is sometimes called {\it Hopf operad}. Let $B\in H_*\calB_d(n)$, $d\geq 2$, be any product of iterated brackets. This cycle is realized by a product of spheres
$$
(S^{d-1})^k\to\calB_d(n).
$$
Thus $\Delta B\in H_*\calB_d(n)\otimes H_*\calB_d(n)$
can be computed from the copruduct of the fundamental class of $(S^{d-1})^k$. For example:
\begin{multline*}
\Delta\left([x_1,x_3]\cdot [x_2,x_4]\right)=
[x_1,x_3]\cdot [x_2,x_4]\otimes x_1\cdot x_2\cdot x_3\cdot x_4 +
[x_1,x_3]\cdot x_2\cdot x_4\otimes x_1\cdot x_3\cdot [x_2,x_4]+\\
(-1)^{d-1} x_1\cdot x_3\cdot [x_2,x_4]\otimes [x_1,x_3]\cdot x_2\cdot x_4+
x_1\cdot x_2\cdot x_3\cdot x_4\otimes [x_1,x_3]\cdot [x_2,x_4].
\end{multline*}
\begin{multline*}
\Delta\left([[x_1,x_3],x_2]\right)=
[[x_1,x_3],x_2]\otimes x_1\cdot x_2\cdot x_3+
[x_1,x_3]\cdot x_2\otimes [x_1\cdot x_3,x_2]+\\
(-1)^{d-1} [x_1\cdot x_3,x_2]\otimes [x_1,x_3]\cdot x_2+
x_1\cdot x_2\cdot x_3\otimes [[x_1,x_3],x_2].
\end{multline*}
Similarly, $H_*\calB_d^{(k)}$ is a {\it Hopf bimodule}. The coproduct of any product of iterated brackets (which is also realized as a map from products of spheres) is computed in the same manner. As example,
\begin{multline*}
\Delta\left[\{x_1\ldots x_k\},\{x_{k+1}\ldots x_{2k}\}\right]=
\left[\{x_1\ldots x_k\},\{x_{k+1}\ldots x_{2k}\}\right]\otimes x_1\cdot\ldots \cdot x_{2k}+\\
\left[\{x_1\ldots x_k\},x_{k+1}\cdot \ldots\cdot x_{2k}\right]\otimes x_1\cdot\ldots \cdot x_{k}\cdot \{x_{k+1}\ldots x_{2k}\}+\\
(-1)^{kd}\left[x_1\cdot\ldots \cdot x_k,\{x_{k+1}\ldots x_{2k}\}\right]\otimes\{x_1\ldots x_k\}\cdot x_{k+1}\cdot \ldots\cdot x_{2k}+ \\
\{x_1\ldots x_k\}\cdot x_{k+1}\cdot \ldots\cdot x_{2k}\otimes \left[x_1\cdot\ldots \cdot x_k,\{x_{k+1}\ldots x_{2k}\}\right]+\\
(-1)^{kd}x_1\cdot\ldots \cdot x_{k}\cdot \{x_{k+1}\ldots x_{2k}\}\otimes \left[\{x_1\ldots x_k\},x_{k+1}\cdot \ldots\cdot x_{2k}\right]+\\
x_1\cdot\ldots \cdot x_{2k}\otimes \left[\{x_1\ldots x_k\},\{x_{k+1}\ldots x_{2k}\}\right].
\end{multline*}
The two summands producing zero were omitted.

Notice that the space of primitives is spanned by the elements that have exactly one long bracket. This space is dual to the space of generators, see Theorem~\ref{t:quadratic}.

\subsection{Cobimodule structure}\label{ss:cobimodule}
The cooperad structure of $H^*\calB_d$ is given by the maps
$$
H^*\calB_d(m_1+\ldots +m_n)\to H^*\calB_d(n)\otimes H^*\calB_d(m_1)\otimes\ldots\otimes H^*\calB_d(m_n)
\eqno(\numb)\label{eq:cocomp}
$$
induced by the composition maps in $\calB_d$. Explicitly, given a forest $T\in H^*\calB_d(m_1+\ldots +m_n)$, $d\geq 2$, the map~\eqref{eq:cocomp} sends it to
$$
T\mapsto \pm (T/{\sim})\otimes T_1\otimes\ldots\otimes T_n,
\eqno(\numb)\label{eq:cocomp_graphs}
$$
where $T_s$ is the restriction of $T$ on the set
$$
M_s=\left\{\sum_{i=1}^{s-1} m_i +1,\sum_{i=1}^{s-1} m_i +2,\ldots, \sum_{i=1}^{s-1} m_i +m_s\right\};
$$
 and $T/{\sim}$ is the quotient of $T$ by the subgraphs $T_s$, $s=1\ldots \ell$. In particular if $T/{\sim}$ has cycles, the result is zero. The sign in~\eqref{eq:cocomp_graphs}  is the Koszul sign due to reordering of the edges of~$T$. This cooperad structure was used for example in~\cite{LV, SinhaWalt1,SinhaWalt2}.

The coaction maps
\begin{eqnarray}
H^*\calB_d^{(k)}(m_1+\ldots +m_n)&\to& H^*\calB_d(n)\otimes H^*\calB_d^{(k)}(m_1)\otimes\ldots\otimes H^*\calB_d^{(k)}(m_n), \label{eq:cocomp_l}\\
H^*\calB_d^{(k)}(m_1+\ldots +m_n)&\to& H^*\calB_d^{(k)}(n)\otimes H^*\calB_d(m_1)\otimes\ldots\otimes H^*\calB_d(m_n)\label{eq:cocomp_r}
\end{eqnarray}
are described by the same formula~\eqref{eq:cocomp_graphs}. In the case of left coaction~\eqref{eq:cocomp_l}, to get non-zero each square vertex of $T$ must be entirely inside one of~$M_s$\rq{}s. For the right coaction~\eqref{eq:cocomp_r}, one obtains non-zero only if at most one element of each square vertex $A$ of $T$ is contained in each of $M_s$:
$$
|A\cap M_s |\leq 1,\qquad s=1\ldots n.
$$

\begin{remark}\label{r:d1cobimod}
In case $d=1$ the coaction has a different description. In fact Baryshnikov's description of $H^*\calB_1^{(k)}(\bullet)$, see Section~\ref{s:d1}, is also nicely compatible with the cobimodule structure over the associative cooperad $H^*\calB_1$.
\end{remark}

\section{Symmetric group action and generating function of dimensions}\label{s:sym_action}
The symmetric group action on the (co)homology of the poset $\Pi_{n,k}$ and on $H_*\calM_2^{(k)}(n)$ was computed in~\cite{SundWachs}. The results can be without any difficulty generalized to any ambient dimension $d$, see Theorem~\ref{t:sym_action} below. Our operadic approach of studying this homology makes the results of~\cite{SundWachs} more transparent. Also the symmetric group action helps to produce an explicit generating function of the Betti numbers, see Corollary~\ref{c:exp_gen_funct}, which seems to be overlooked in the literature and is given here for completeness of exposition.

The symmetric sequences of graded vector spaces form a monoidal category with respect to the composition operation~$\circ$ and unit~$\one$~\cite{LodayVallette}. If we are working over a field any symmetric sequence $M(n)$, $n\geq 0$, defines a functor $M\colon Vect\to Vect$ that sends a vector space $V\mapsto \oplus_{n=0}^\infty M(n)\otimes_{\Sigma_n} V^{\otimes n}$.  The composition is defined in such a way that $(M\circ N)(V)=M(N(V))$. In fact one does not need the base ring to be a field in order to define this composition. The unit $\one$ for this operation is the sequence which is zero in all arities except one and it is the base ring in arity one.  Notice that $\one\colon Vect\to Vect$ is the identity functor. The construction works nicely over integeres: in case $M$ and $N$ are torsion free and $N(0)=0$, the composition $M\circ N$ is also torsion free. For a graded vector space $V=\oplus_{n\in\Z} V_n$ we will define its {\it graded dimension} as a formal power series in $q$:
$$
\dim V=\sum_n \dim V_n\cdot q^n.
$$
For a symmetric sequence $M$ of graded vector spaces we define the {\it exponential generating function} of its components
$$
F_M(x)=\sum_{j=0}^{+\infty} \dim M(j) \frac{x^j}{j!}.
$$
One has
$$
F_{M\circ N}(x)=F_M(F_N(x)).
\eqno(\numb)\label{eq:comp_gen}
$$
For a symmetric sequence $M$ denote by $M\{d\}$ its operadic $d$-suspension. As a vector space $M\{d\}(n)$ is $d(n-1)$-times suspended space $M(n)$. As a $\Sigma_n$-module $M\{d\}(n)\simeq M(n)\otimes(sign_n)^{\otimes d}$, where $sign_n$ is the sign representation of $\Sigma_n$. It is straightforward that
$$
F_{M\{d\}}(x)=\frac 1{q^d} F_M(q^dx).
\eqno(\numb)\label{eq:susp_gen}
$$
Notice also that
$$
(M\circ N)\{d\}=\left(M\{d\}\right)\circ\left(N\{d\}\right).
\eqno(\numb)\label{eq:comp_susp}
$$

To recall $\Com$ denotes the operad of commutative unital algebras and $\Lie$ denotes the operad of Lie algebras -- both viewed as symmetric sequences over $\Z$. One has
\begin{eqnarray}
F_{\Com}(x)&=&e^x;
\label{eq:comm_gen}\\
F_{\Lie}(x)&=&-\ln(1-x).
\label{eq:lie_gen}
\end{eqnarray}
\sloppy Let $\calH_d^{(k)}(n)\subset H_*\calB_d^{(k)}(n)$  be the subspace spanned by elements of the form $[\ldots[\{x_{\sigma_1}\ldots x_{\sigma_k}\},x_{\sigma_{k+1}}]\ldots x_{\sigma_n}]$ (in other words spanned by the iterated brackets that have only one long bracket). The operadic $(d-1)$-desuspension $\calH_d^{(k)}\{1-d\}$ of this symmetric sequence does not depend on $d$ and will be denoted by $\calH_1^{(k)}$. It follows from Proposition~\ref{p:d>1_basis} that $\calH_1^{(k)}(n)$ is concentrated in grading $(k-2)$ and has dimension ${n-1}\choose{k-1}$. One has

$$
F_{\calH_1^{(k)}}(x)=\frac{q^{k-2}x^k}{(k-1)!}\sum_{j=0}^{+\infty}\frac{x^j}{(j+k)\cdot j!}=\\
(-q)^{k-2}-(-q)^{k-2}\left(\sum_{j=0}^{k-1}\frac{(-x)^j}{j!}\right)e^x.
\eqno(\numb)\label{eq:exp_gen_h1}
$$
The last equality was obtained by noticing that
$$
F_{\calH_1^{(k)}}\rq{}(x)=\frac{q^{k-2}}{(k-1)!}x^{k-1}e^x;
\eqno(\numb)\label{eq:exp_gen_h1_pr}
$$
and then integrating.

\begin{lemma}\label{l:hook}
For any $n\geq k\geq 2$ one has an isomorphism of $\Z[\Sigma_n]$-modules
$$
\calH_1^{(k)}(n)\simeq \Z[\Sigma_n]\cdot a\cdot b,
$$
where $a=\sum_{\sigma\in\Sigma_k}(-1)^{|\sigma|}\sigma$,
and $b=\sum_{\sigma\in\Sigma_{ \{1,k+1,k+2,\ldots,n\} }}\sigma$.
\end{lemma}

In particular this lemma says that $\calH_1^{(k)}(n)\otimes\Q$ is the irreducible representation of  {\it hook type} $(n-k+1,k)$, see~\cite{FultHarris}.

\begin{proof}
We define a map $\calH_1^{(k)}(n)\to \Z[\Sigma_n]\cdot a\cdot b$ by sending $[\ldots[\{x_1\ldots x_k\},x_{k+1}],\ldots x_n]\mapsto e\cdot a\cdot b$, where $e\in\Sigma_n$ is the unit element. One has to check that this map is correctly defined. First we notice that any element $\sigma\in\Sigma_k$ acts both on
$[\ldots[\{x_1\ldots x_k\},x_{k+1}]\ldots x_n]$ and on $e\cdot a\cdot b$ as multiplication by $(-1)^\sigma$. Also any $\sigma\in\Sigma_{\{k+1,k+1,\ldots,n\}}$ acts as identity on both of them. And finally an easy verification shows that relation~\eqref{eq:gen_jacob} is also satisfied. On the other hand the map is obviously surjective. The fact that the target has the same dimension ${n-1}\choose{k-1}$ as the source  ensures that the map is an isomorphism.
\end{proof}

\begin{remark}\label{r:hook}
Let $\calH_1^{(k)}(n)^\vee$ denote the dual $\Sigma_n$-module that we described as a space of $k$-trees with a single square vertex and quotiented out by relations~\eqref{eq:dual_jacobi}. Looking at the generalized Jacobi~\eqref{eq:gen_jacob} and the relations~\eqref{eq:dual_jacobi}   it is easy to see that one has an obvious isomorphism of $\Sigma_n$-modules
$$
\calH_1^{(k)}(n)^\vee\simeq \calH_1^{(n-k+1)}(n)\otimes sign_n.
$$
This implies that one has a $\Z[\Sigma_n]$-module isomorphism
$$
(\calH_1^{(k)}(n))^\vee\simeq \Z[\Sigma_n]\cdot b\cdot a,
$$
where $a$ and $b$ are from Lemma~\ref{l:hook}.\footnote{Of course rationally a $\Sigma_n$-module is always isomorphic to its dual: $\Q[\Sigma_n]\cdot a\cdot b\simeq \Q[\Sigma_n]\cdot b\cdot a$.}
\end{remark}

\begin{theorem}[\cite{SundWachs}]\label{t:sym_action}
For $d\geq 2$, $k\geq 3$, one has a natural isomorphism of  symmetric sequences
$$
H_*\calB_d^{(k)}\simeq \Com\circ\left(\one\oplus(\Lie\circ\calH_1^{(k)})\{d-1\}\right).
\eqno(\numb)\label{eq:sym_action}
$$
For $d=1$ and/or $k=2$ this isomorphism holds over $\Q$.
\end{theorem}

\begin{proof}
In case $d\geq 2$, $k\geq 3$, one has that $H_*\calB_d^{(k)}$ is a left module over $H_*\calB_d=\Com\circ\left(\Lie\{d-1\}\right)$ and $\calH_d^{(k)}(\bullet)\simeq\calH_1^{(k)}\{d-1\}(\bullet)$ is a sequence of subobjects in $H_*\calB_d^{(k)}(\bullet)$. This left action defines  a map
$$
\Com\circ\left(\one\oplus(\Lie\{d-1\})\circ\calH_d^{(k)}\right)\to H_*\calB_d^{(k)},
$$
where $\one$ corresponds to $H_0\calB_d^{(k)}(1)\simeq\Z$.  Proposition~\ref{p:d>1_basis} ensures that this map is an isomorphism.

In case $k=2$, the right-hand side of~\eqref{eq:sym_action} is isomorphic to the associated graded quotient of $H_*\calB_d^{(2)}$ by a similar argument and by Remark~\ref{r:case_k2}, see also Remark~\ref{r:case_k2_coh}. Since over $\Q$ any filtration of $\Sigma_n$-modules splits, we get the result.

Similarly for $d=1$, the operad $H_*\calB_1=\Assoc$ admits a natural increasing (Poincar\'e-Birkhoff-Witt) filtration, whose associated graded quotient is the Poisson operad. The aforementioned filtration is compatible with a filtration in the left module $H_*\calB_1^{(k)}$. The associated graded quotient of the latter symmetric sequence is the right-hand side of ~\eqref{eq:sym_action}.

In case $k=2$ and $d=1$ one has to take the associated graded quotient twice.
\end{proof}

\begin{remark}\label{r:lie_decomp}
In particular we get an isomorphism of symmetric sequences
$$
\one\oplus\Lie\circ\calH_1^{(2)}\simeq_\Q\Lie,
\eqno(\numb)\label{eq:reut}
$$
which at first might appear surprising, but it simply means that for any (graded) vector space $V$, the Lie subalgebra $\Lie_{\geq 2}(V)$ (spanned by Lie monomials of degree $\geq 2$) of the free Lie algebra $\Lie(V)$ (generated by $V$) is isomorphic to the free Lie algebra generated by $\calH_1^{(2)}(V)=\oplus_{n\geq 2}\calH_1^{(2)}(n)\otimes_{\Sigma_n}V^{\otimes n}$. This is a particular occurence of a general fact that a Lie subalgebra of a free Lie algebra is always free~\cite{Reut_book}. The isomorphism~\eqref{eq:reut} is actually also due to Reutenauer~\cite{Reut_char}.
\end{remark}

\begin{corollary}\label{c:exp_gen_funct}
The exponential generating function of graded dimensions for the symmetric sequence $H_*\calB_d^{(k)}(\bullet)$ is as follows
$$
F_{H_*\calB_d^{(k)}}(x)=e^x\left(1-(-q)^{k-2}+(-q)^{k-2}\left(\sum_{j=0}^{k-1}\frac{(-q^{d-1}x)^j}{j!}\right)e^{q^{d-1}x}\right)^{-\frac 1{q^{d-1}}}.
\eqno(\numb)\label{eq:exp_gen_funct1}
$$
\end{corollary}

\begin{remark}\label{r:exp_gen_funct}
For explicit computations of the Betti numbers it is more convenient to use the formula
$$
F_{H_*\calB_d^{(k)}}(x)=e^x\left(1-\frac{q^{kd-2}x^k}{(k-1)!}\sum_{j=0}^{+\infty}\frac{(q^{d-1}x)^j}{(j+k)\cdot j!}\right)^{-\frac 1{q^{d-1}}}.
\eqno(\numb)\label{eq:exp_gen_funct2}
$$
\end{remark}

\begin{proof}[Proof of Corollary~\ref{c:exp_gen_funct}]
It  is a consequence of Theorem~\ref{t:sym_action} together with \eqref{eq:comp_gen}, \eqref{eq:susp_gen},  \eqref{eq:comm_gen}, \eqref{eq:lie_gen}, \eqref{eq:exp_gen_h1}.
\end{proof}

\begin{remark}\label{r:betti_numbers}
The Betti numbers for $\calM_d^{(k)}(n)$ were computed in~\cite{BjorWelk}, see also~\cite{PRW}. The formulae~\eqref{eq:exp_gen_funct1}, \eqref{eq:exp_gen_funct2} provide a more compact way to keep track of these data.
\end{remark}

\section{Application: spaces of non-$k$-equal immersions}\label{s:application}
This section stays very separately from the rest of the paper. Its goal is to show that the considered bimodules appear very naturally in Topology, and what we explain here is just one of its applications. Theorems~\ref{t:cosimpl}, and \ref{t:inf_bim}-\ref{t:r_mod} below were proved for embedding spaces in~\cite{Tur_these}, and~\cite{AT}, respectively. We just want to point out that the proofs are completely analogous for spaces of non-$k$-equal immersions.

Let $M$ be an open subset of $\R^m$, and $n>m$. Consider the space $\Imm^{(k)}(M,\R^n)$ of immersions $f\colon M\looparrowright\R^n$ such that for any cardinality $k$ subset $K\subset M$, one has that $f|_K$ is non-constant. We call such maps {\it non-$k$-equal immersions}.  For example the space $\Imm^{(2)}(M,\R^n)$ is the space of embeddings $\Emb(M,\R^n)$.

Let $\Imm(M,\R^n)$ denote the space of immersions, and let $\Ibar^{(k)}(M,\R^n)$ be the homotopy fiber of the natural inclusion
$$
\Imm^{(k)}(M,\R^n)\hookrightarrow\Imm(M,\R^n)
$$
over the composition $M\subset \R^m\subset\R^n$.

We will also consider spaces $\Imm_c^{(k)}(\R^m,\R^n)$ of {\it long} non-$k$-equal immersions, where the subscript $c$ stays for {\it compact support}. Points of this space are non-$k$-equal immersions $\R^m\looparrowright\R^n$ coinciding with the fixed linear inclusion $\R^m\subset\R^n$ outside a compact subset of~$\R^m$. One gets a similar fiber sequence
$$
\Ibar_c^{(k)}(\R^m,\R^n)\to\Imm_c^{(k)}(\R^m,\R^n)\to\Imm_c(\R^m,\R^n).
\eqno(\numb)\label{eq:f_seq}
$$
The Smale-Hirsch principle~\cite{Hirsch} provides us with natural equivalences
\begin{gather}
\Imm_c(\R^m,\R^n)\simeq \Omega^mV_{m,n};
\label{eq:st1}
\\
\Imm(M,\R^n)\simeq \mathrm{Maps}(M,V_{m,n}),
\label{eq:st2}
\end{gather}
where $V_{m,n}$ is the Stiefel manifold of isometric linear maps $\R^m\hookrightarrow\R^n$.

The reason we study $\Ibar^{(k)}(M,\R^n)$ and $\Ibar_c^{(k)}(\R^m,\R^n)$ is that their homotopy type and homology have nice properties in comparison with the initial spaces of non-$k$-equal immersions. But at the same time they differ from $\Imm^{(k)}(M,\R^n)$ and $\Imm_c^{(k)}(\R^m,\R^n)$ by an easily controllable term~\eqref{eq:st1}, \eqref{eq:st2}.

There are two main approaches to study such functional spaces. The first approach, due to Vassiliev and usually called {\it Theory of Discriminants}~\cite{Vass_discr}, consists in considering the space of all smooth maps from our manifold to $\R^n$. This space is an affine space of infinite dimension and thus contractible. The cohomology classes of the space of maps that avoid any given types of singularities are described via linking number with cycles (of finite codimension) in the complement space called {\it discriminant} that consists of singular maps. The discriminant is a semi-algebraic set whose stratification provides the necessary combinatorial information to compute the homology of the complement.  The second approach, called {\it manifold Calculus} was developped by Goodwillie and Weiss~\cite{GoodWeiss,WeissEmb}. This second approach was mostly used to study spaces of embeddings, but it can also be used to study more general functional spaces. For this approach instead of looking on maps from $M$ to $N$ (avoiding given multi-singularities) one varies the source to be any open subset $U\subset M$. This produces a presheaf  on $M$ in topological spaces. In some cases the obtained presheaf is a homotopy sheaf, for example it is the case for spaces of immersions, but in general it is not true. Homotopy sheaves are {\it linear} functors from the point of view of {\it Manifold Calculus}. But there are also quadratic, cubical, and more generally polynomial of any degree $k$ presheaves, which also mean that they have some nice \lq\lq from local to global\rq\rq{} properties. The manifold calculus assigns to any topological presheaf on $M$ a {\it Taylor tower} of its polynomial approximations:
$$
\xymatrix{
&F\ar[dl] \ar[d] \ar[dr] \ar[drr]\\
T_0F&T_1F\ar[l]&T_2F\ar[l]&T_3F\ar[l]&\ldots\ar[l]
}
\eqno(\numb)\label{eq_tower}
$$
In good cases the limit of the tower $T_\infty F$ is equivalent to $F$.

We believe that Vassiliev\rq{}s theory of discriminants can also be expressed in terms of the manifold calculus by describing the discriminant set as a spectrum Spanier-Whitehead dual to the given space of non-singular maps. (Here one has to consider the copresheaf that assigns to $U$ the corresponding spectrum. Notice that one will need to use the covariant version of the calculus instead of the contravariant one usually used.) This construction would prove an equivalence of two approaches. There is a work in this direction~\cite{ReisWeiss}, but in general this equivalence has not been established yet.

Both methods produce spectral sequences computing the homology and the first term of the Vassiliev spectral sequence is isomorphic to the second term of the manifold calculus homology spectral sequence.

On the other hand, the manifold calculus can be translated into operadic language~\cite{AT,DeBrito-Weiss,Tur_FM}. We explain below how this interpretation is applied to the spaces $\Ibar^{(k)}_c(\R^m,\R^n)$, $\Ibar^{(k)}(M,\R^n)$.

As we have seen in Section~\ref{s:left_bim}, $H_*\calB_n^{(k)}$ is a bimodule under $H_*\calB_n$. Inclusion $\R^1\subset\R^n$ induces inclusion of operads $\calB_1\hookrightarrow\calB_n$, which produces a map of operads in homology:
$$
\Assoc\to H_*\calB_n.
$$
Due to this morphism $H_*\calB_n^{(k)}$ is also a bimodule under $\Assoc$, which endows  $H_*\calB_n^{(k)}$ with a cosimplicial structure.\footnote{One uses compositions with the product $x_1x_2\in \Assoc(2)$ to get coface maps, and compositions with the unit $1\in\Assoc(0)$ to get codegeneracies.}

\begin{theorem}\label{t:cosimpl}
The first term of the Vassiliev spectral sequence and the second term of the manifold calculus homology spectral sequence computing $H_*\Ibar^{(k)}_c(\R^1,\R^n)$, $nk\geq 6$, is isomorphic to the homology of the total cosimplicial complex $\Tot H_*\calB_N^{(k)}(\bullet)$.
\end{theorem}

For the manifold calculus approach this statement is a particular instance of Theorem~\ref{t:inf_bim} below. (See also~\cite{Sinha_cosimpl} whose construction applies only to the case of embeddings~$k=2$.)  For the Vassiliev method, one has to consider the space $\Ibar^{(k)}_c(\R^1,\R^n)$ as an open subset in the space of all smooth maps $[0,1]\times\R\to\R^n$ with the restriction $f(t,x)=(x,0,0,\ldots,0)$ outside a compact subset of $[0,1)\times\R$, as in~\cite{Tur_one_term}.

In order to formulate a higher dimensional analogue of the theorem above, we need to recall some terminology from the theorey of operads.

An {\it infinitesimal bimodule} over an operad $\calO$ is a sequence of objects $M=\{M(n),\, n\geq 0\}$ (symmetric sequence in case $\calO$ is a $\Sigma$-operad, or just a seqence in case $\calO$ is non-$\Sigma$), endowed with composition maps:
\begin{align*}
\circ_i\colon\calO(n)\otimes M(k)\to M(n+k-1);& \text{ (infinitesimal left action)}\\
\circ_i\colon M(n)\otimes \calO(k)\to M(n+k-1).& \text{ (infinitesimal right action)}
\end{align*}
These composition maps have to satisfy natural unity, associativity, and $\Sigma$-compatibility conditions~\cite{LodayVallette,MerkVal,Tur_Hodge}. For example an infinitesimal bimodule over the non-$\Sigma$ associative operad is exactly the same thing as a cosimplicial object.

Notice that infinitesimal right action is equivalent to the usual right action since one can use the identity element $\id\in\calO(1)$ to mimic empty insertions. But  infinitesimal left action is
 essentially different from the usual left action. Moreover neither of them can be obtained one from another. However in case $M$ is a bimodule under $\calO$, i.e. $M$ is a bimodule over
 $\calO$ endowed with a map of $\calO$-bimodules $\rho\colon\calO\to M$, then $M$ inherits the structure of an infinitesimal bimodule.\footnote{One uses $\rho(\id)$  to mimic
 empty insertions.} Thus $\calB_n^{(k)}$ is an infinitesimal bimodule over $\calB_n$ and also over $\calB_m$, $m<n$, by restriction.

Theorem~\ref{t:inf_bim} appeared in~\cite{AT} for spaces of embeddings. The proof works also in our situation.

\begin{theorem}[\cite{AT}]\label{t:inf_bim}
The limit of the Goodwillie-Weiss tower for the space $\Ibar_c^{(k)}(\R^m,\R^n)$, $n>m$, is weakly equivalent to the space of derived maps of infinitesimal bimodules over~$\calB_m$:
$$
T_\infty\Ibar_c^{(k)}(\R^m,\R^n)\simeq \hIBim_{\calB_m}(\calB_m,\calB_n^{(k)}).
\eqno(\numb)\label{eq:inf_bim1}
$$
The same is true for singular chains
$$
T_\infty C_*\Ibar_c^{(k)}(\R^m,\R^n)\simeq \hIBim_{C_*\calB_m}(C_*\calB_m,C_*\calB_n^{(k)}).
\eqno(\numb)\label{eq:inf_bim2}
$$
\end{theorem}

For a codimesnion zero submanifold $M\subset \R^m$, denote by $\sEmb(\bullet,M)$ the symmetric sequence $\sEmb(\sqcup_n D^m,M)$, $n\geq 0$, where $\sEmb$ stays for the space of {\it standard} embeddings which on each connected component are compositions of translations and rescalings.
Notice that $sEmb(\bullet,M)$ is naturally a right module over $\calB_m$. The theorem below is a particular case of the enriched version of the manifold calculus.

\begin{theorem}[\cite{AT,DeBrito-Weiss,Tur_FM}]\label{t:r_mod}
For any open $M\subset\R^m$ the limit for the Goodwillie-Weiss tower for the space $\Ibar^{(k)}(M,\R^n)$, $n>m$, is weakly equivalent to the space of derived maps of right modules over $\calB_m$:
$$
T_\infty\Ibar^{(k)}(M,\R^n)\simeq \hRmod_{\calB_m}(\sEmb(\bullet,M),\calB_n^{(k)}).
\eqno(\numb)\label{eq:r_mod1}
$$
The same is true for singular chains
$$
T_\infty C_*\Ibar^{(k)}(M,\R^n)\simeq \hRmod_{C_*\calB_m}(C_*\sEmb(\bullet,M),C_*\calB_n^{(k)}).
\eqno(\numb)\label{eq:r_mod2}
$$
\end{theorem}

The convergence of the towers~\eqref{eq:inf_bim1}, \eqref{eq:inf_bim2}, \eqref{eq:r_mod1}, \eqref{eq:r_mod2} to the initial spaces or chain complexes has not been studied yet. This question is actually very difficult.

The second parts of Theorems~\ref{t:inf_bim},~\ref{t:r_mod} imply that there are natural spectral sequences computing $H_*\Ibar_c^{(k)}(\R^m,\R^n)$, $H_*\Ibar^{(k)}(M,\R^n)$ (manifold calculus homology spectral sequences) whose first terms together with their differentials are described using the infinitesimal $H_*\calB_m$-bimodule structure of $H_*\calB_n^{(k)}$.

Theorem~\ref{t:r_mod} has a version where $M$ is any manifold and not necessarily an open subset of $\R^m$. In the latter case one has to use the framed discs operad instead as well as the  framed version of $\calB_n^{(k)}$, see~\cite{DeBrito-Weiss,Tur_FM}.

We finish this paper by mentioning that the fact that $\calB_n^{(k)}$ is a bimodule under $\calB_m$ (and not only an infinitesimal bimodule) governs the $\calB_m$-algebra structure on $T_\infty\Ibar_c^{(k)}(\R^m,\R^n)$. The following result was announced by Dwyer and Hess~\cite{DwHe2}:

\begin{theorem}[Dwyer-Hess \cite{DwHe2}]\label{t:deloop}
Let $\calM$ be a bimodule under $\calB_m$ satisfying $\calM(0)\simeq\calM(1)\simeq *$, then
$$
\hIBim_{\calB_m}(\calB_m,\calM)\simeq\Omega^m\hBim_{\calB_m}(\calB_m,\calM).
$$
\end{theorem}

The right-hand side $\hBim(-,-)$ above denotes the space of derived maps of bimodules. $\Omega^m$ denotes as usual the $m$-iterated loop space, where for a base point one takes the structure map $\calB_m\to M$. In case $m=1$ this theorem was proved in~\cite{DwHe1,Tur_deloop}. In case $\calM$ is an operad endowed with a map $\calO\to\calM$ (which enables $\calM$ with a structure of a bimodule under $\calO$), one can get an extra delooping
$$
\hIBim_{\calB_m}(\calB_m,\calO)\simeq\Omega^{m+1}\hOperad(\calB_m,\calM).
$$

This equivalence corresponds to the fact that  the space $\Ebar_c(\R^m,\R^n)$ has a structure of a $\calB_{m+1}$-algebra thanks to the fact that one can pull one knot through the other~\cite[Corollary~7]{Budney}, \cite[Proposition~1.1]{Tur_Hodge}. But the space $\Ibar^{(k)}_c(\R^m,\R^n)$, $k\geq 3$ is only a $\calB_m$-algebra --- given two long non-$k$-equal immersions, pulling one such map through the other  is impossible in general since it might produce forbidden singularities.

\end{document}